%% file: RDTP.tex
\documentclass{amsart}

\include{RDTP-Preamble}

\begin{document}

\title{The balanced tensor product of module categories}

\begin{abstract}
The balanced tensor product $M \otimes_A N$ of two modules over an algebra $A$ is the vector space corepresenting $A$-balanced bilinear maps out of the product $M \times N$.  The balanced tensor product $\cM \boxtimes_{\cC} \cN$ of two module categories over a monoidal linear category $\cC$ is the linear category corepresenting $\cC$-balanced right-exact bilinear functors out of the product category $\cM \times \cN$.  We show that the balanced tensor product can be realized as a category of bimodule objects in $\cC$, provided the monoidal linear category is finite and rigid.
\end{abstract}


\author{Christopher L. Douglas}
\address{Mathematical Institute\\ University of Oxford\\ Oxford OX1 3LB\\ United Kingdom}
\email{cdouglas@maths.ox.ac.uk}
\urladdr{http://people.maths.ox.ac.uk/cdouglas}
      	
\author{Christopher Schommer-Pries}
\address{Department of Mathematics\\ Max Planck Institute for Mathematics \\ 53111 Bonn \\ Germany}
\email{schommerpries.chris.math@gmail.com}
\urladdr{http://sites.google.com/site/chrisschommerpriesmath}

\author{Noah Snyder}
\address{Department of Mathematics\\ Indiana University\\ Bloomington, IN 47401\\ USA}
\email{nsnyder@math.indiana.edu}
\urladdr{http://www.math.columbia.edu/\!\raisebox{-1mm}{~}nsnyder/}

\thanks{CD was partially supported by a Miller Research Fellowship and by EPSRC grant EP/K015478/1, CSP was partially supported by NSF fellowship DMS-0902808 and the Max Planck Institute for Mathematics, and NS was partially supported by NSF fellowship DMS-0902981 and DARPA HR0011-11-1-0001.  We thank the referee for suggestions and corrections, and thank Manuel Ara\'ujo for pointing out a simplification of the proof of Theorem~\ref{thm:DelignePrdtOverATCExists}.
}

\maketitle	

\tikzexternaldisable


Tensor categories are a higher dimensional analogue of algebras.  Just as modules and bimodules play a key role in the theory of algebras, the analogous notions of module categories and bimodule categories play a key role in the study of tensor categories (as pioneered by Ostrik \cite{MR1976459}).  One of the key constructions in the theory of modules and bimodules is the relative tensor product $M \otimes_A N$.  As first recognized by Tambara \cite{tambara}, a similarly important role is played by the balanced tensor product of module categories over a monoidal category $\cM \boxtimes_\cC \cN$ (for example, see \cite{MR1966524, 0909.3140, MR2511638, MR2909758, 1202.4396, MR3022755, MR3063919}).  For $\cC = \Vect$, this agrees with Deligne's~\cite{MR1106898} tensor product $\cK \boxtimes \cL$ of finite linear categories. Etingof, Nikshych, and Ostrik~\cite{0909.3140} established the existence of a balanced tensor product $\cM \boxtimes_{\cC} \cN$ of finite semisimple module categories over a fusion category, and Davydov and Nikshych \cite[\S 2.7]{MR3107567} outlined how to generalize this construction to module categories over a finite tensor category.\footnote{We were unaware of Davydov--Nikshych's result while writing this paper.}  We give a new construction of the balanced tensor product over a finite tensor category as a category of bimodule objects.

Recall that the balanced tensor product $M \otimes_A N$ of modules is, by definition, the vector space corepresenting $A$-balanced bilinear functions out of $M \times N$.  In other words, giving a map $M \otimes_A N \rightarrow X$ is the same as giving a billinear map $f: M \times N \rightarrow X$ with the property that $f(ma,n) = f(m, an)$.  If the balanced tensor product exists, it is certainly unique (up to unique isomorphism), but the universal property does not guarantee existence.  Instead existence is typically established by an explicit construction as a quotient of a free abelian group on the product $M \times N$. We now describe the balanced tensor product $\cM \boxtimes_{\cC} \cN$ of module categories over a tensor category.  Again this should be universal for certain bilinear functors, however when passing from algebras to tensor categories, the analogue of the equality $f(ma,n) = f(m, an)$ is a natural system of isomorphisms $\eta_{m,a,n}: \cF(m \otimes a,n) \rightarrow \cF(m, a \otimes n)$ satisfying some natural coherence properties.  A bilinear functor $\cF$ together with a natural coherent system of isomorphisms $\eta_{m,a,n}$ is called a $\cC$-balanced functor.  Thus the balanced tensor product $\cM \boxtimes_{\cC} \cN$ is defined to be the linear category corepresenting $\cC$-balanced right-exact bilinear functors out of $\cM \times \cN$. In other words, giving a right-exact functor $\cM \boxtimes_{\cC} \cN \rightarrow \cX$ is the same as giving a right-exact bilinear functor $\cM \times \cN \rightarrow \cX$ together with isomorphisms $\eta_{m,a,n}: \cF(m \otimes a,n) \rightarrow \cF(m, a \otimes n)$ satisfying a coherence property.  Again if the balanced tensor product exists, it is unique (up to an equivalence that is unique up to unique natural isomorphism), but existence does not follow formally. 

The only task is to provide an explicit construction of the balanced tensor product that satisfies the corepresentation property.  A crucial tool for our construction is the result, proved by Etingof--Gelaki--Nikshych--Ostrik \cite[Cor. 7.10.5]{egno-book}, based on earlier work of Etingof--Ostrik \cite[\S 3.2]{EO-ftc} along the lines pioneered by Ostrik \cite[Thm 1]{MR1976459}, that any finite left module category ${}_{\cC} \cN$ over a finite tensor category $\cC$ is equivalent to the category of right module objects $\Mod{}{B}(\cC)$ in $\cC$, for some algebra object $B \in \cC$; similarly any right module category $\cM_{\cC}$ is equivalent to a category of left module objects $\Mod{A}{}(\cC)$.  

\begin{theorem*}
Let $\cC$ be a finite, rigid, monoidal linear category, and let $\cM$ be a finite right $\cC$-module category and $\cN$ a finite left $\cC$-module category.  Let $A$ and $B$ be algebra objects in $\cC$ such that $\cM \cong \Mod{A}{}(\cC)$ and $\cN \cong \Mod{}{B}(\cC)$ as module categories.  The category $\Mod{A}{B}(\cC)$ of $A$--$B$-bimodule objects in $\cC$ corepresents $\cC$-balanced right-exact bilinear functors out of $\cM \times \cN$, and therefore realizes the balanced tensor product $\cM \boxtimes_\cC \cN$.
\end{theorem*}

The balanced tensor product will play a key role in our study of the $3$-category of finite tensor categories and the associated local toplogical field theories in \cite{DTCI}.  Indeed, we use at some crucial steps not only that the balanced tensor product exists, but that it can be explicitly realized as $\Mod{A}{B}(\cC)$.  Our realization also shows that the balanced tensor product of finite linear (abelian) module categories (the ``Deligne" tensor, cf.~\cite{MR1106898}) is equivalent to the balanced tensor product of finitely cocomplete $k$-additive module categories (the ``Kelly" tensor, cf.~\cite{MR651714, MR648793}).  This equivalence will allow us, in \cite{DTCI}, to work with abelian categories while relying on results from a merely additive context, particularly the Johnson-Freyd--Scheimbauer~\cite{MR3590516} construction of a 3-category of (finitely cocomplete $k$-additive) tensor categories.

The first two sections of this paper are largely expository and intended as a self-contained introduction to the parts of EGNO's theory that are needed in our other papers.  In Section~\ref{sec:tc-lincat}, we give an overview of linear categories and monoidal linear categories.  In particular, we give a proof of the well known fact that all finite linear categories are equivalent to categories of modules over a finite-dimensional algebra, and a proof of the fact that all right (respectively left) exact linear functors between finite linear categories admit right (respectively left) adjoints.  In Section~\ref{sec:tc-bimod}, we provide a review of aspects of Etingof--Gelaki--Nikshych--Ostrik's theory of module categories, including proofs that all module functors are strong and that adjunctions lift to module adjunctions, and we give a proof, of the fact that module categories are categories of module objects, that isolates and mitigates the rigidity assumption.  In Section~\ref{sec:tc-deligne}, we use the presentation of module categories as categories of modules to prove that the balanced tensor product of module categories exists, and we establish the basic exactness properties of the balanced tensor product.

\section{Linear categories and tensor categories} \label{sec:tc-lincat}

\subsection{Finite linear categories}

	Let $k$ be a fixed ground field, let $\overline{\Vect}_k$ be the category of (possibly infinite-dimensional) $k$-vector spaces, and let $\Vect_k$ be the category of finite-dimensional $k$-vector spaces.   A {\em linear category} is an abelian category with a compatible enrichment over $\overline{\Vect}_k$. 
A {\em linear functor} is an additive functor, that is also a functor of $\overline{\Vect}_k$-enriched categories. 

\begin{warning}
	In \cite{DTCI} we will use the phrase ``linear functor" to mean what we call ``right exact linear functor" in this paper.  In the $3$-category of finite tensor categories the $2$-morphisms are assumed to be right exact, because the balanced tensor product of linear categories is only functorial with respect to right exact functors.  Since this paper concerns the definition of the balanced tensor product itself, we will not use the abbreviated convention here.
\end{warning}

Recall the following standard terminology:
\begin{itemize}
	\item[-] An object of a linear category is {\em simple} if it admits no non-trivial subobjects. The endormorphism ring of any simple object is a division algebra over $k$. 
	\item[-] An object $X$ of a linear category has {\em finite length} if every strictly decreasing chain of subobjects $X = X_0 \supsetneq X_1 \supsetneq X_2 \supsetneq  \cdots$ has finite length. 
	\item[-] A linear category is {\em semisimple} if every object splits as a direct sum of simple objects. If in addition every object has finite length, then every object splits as a finite direct sum of simple objects.
	\item[-] A linear category {\em has enough projectives} if for every object $X$, there is a projective object $P$ with a surjection $P \twoheadrightarrow X$. 
\end{itemize}

\begin{definition} 
	A linear category $\cC$ is {\em finite} if 
	\begin{enumerate}
		\item[1.] $\cC$ has finite-dimensional spaces of morphisms;
		\item[2.] every object of $\cC$ has finite length;
		\item[3.] $\cC$ has enough projectives
		; and
		\item[4.] there are finitely many isomorphism classes of simple objects.  
	\end{enumerate}
\end{definition}

\begin{example}
	The category of finite dimensional vector spaces is a finite linear category. 
\end{example}

The following proposition is well-known (see for instance \cite[\S9.6]{MR2808160}), and justifies the above definition of finite.

\begin{proposition} \label{prop:finitelinearcatsasmodules}
A linear category is finite if and only if it is equivalent to the category $\Mod{A}{}$ of finite-dimensional modules over a finite-dimensional $k$-algebra $A$.
\end{proposition}

We will first prove two lemmas.

\begin{lemma}\label{lem:fiaghtfulrightadjoint}
	Let $F: \cC \leftrightarrows \cD: U$ be an adjunction between abelian categories. Then the right adjoint $U$ is faithful if and only if the counit $FU(X) \to X$ is a surjection for every object $X \in \cD$. If in addition $U$ is exact, then $U$ reflects isomorphisms. 
\end{lemma}

\begin{proof}
	The functor $U$ is faithful precisely when $U(f) = 0$ implies $f=0$ for all morphisms $f: X \to Y$ in $\cD$. Suppose that the counit $\varepsilon_X: FU(X) \to X$ is a surjection for every object $X \in \cD$ and let $f: X \to Y$ be a morphism such that $U(f): U(X) \to U(Y)$ is the zero morphism. Then the composite $FU(X) \to FU(Y) \to Y$ is the zero morphism. Since the counit is natural, this is the same as the composite $FU(X) \to X \to Y$, hence this composite is also the zero morphism. Now since $FU(X) \to X$ is surjective, the original map $f: X \to Y$ must be the zero morphism. 
	
	In the other direction, suppose that $U$ is faithful, and fix an object $X \in \cD$. Let $f: X \to C$ be the cokernel of the counit map $\varepsilon_X: FU(X) \to X$. We wish to show that the cokernel is zero.  Since the composite $f \circ \varepsilon_X =0 $, we have $U(f) \circ U(\varepsilon_X) = 0$. However $U(\varepsilon_X): UFU(X) \to U(X)$ is split (by the unit $\eta_{UX}$ of the adjunction) and hence is surjective, which implies that $U(f) = 0$. Since $U$ is faithful, we have that $f=0$ and so the cokernel was, in fact, zero as desired.   
	
For the last statement, let $f: X \to Y$ be a map, with kernel and cokernel sequence $K \to X \to Y \to C$, and assume that $U(f): U(X) \to U(Y)$ is an isomorphism. Then, since $U$ is exact, the maps $U(K) \to U(X)$	    and $U(Y) \to U(C)$ are zero maps. Hence, since $U$ is faithful, the kernel and cokernel of $f$ were zero maps. 
\end{proof}

\begin{lemma}\label{lem:recognizefinitecats}
	Let $F: \cC \leftrightarrows \cD: U$ be an adjunction between linear categories in which $U$ and $F$ are linear functors, and where $U$ is exact and faithful. Suppose that $\cC$ is finite, then $\cD$ is also finite. 
\end{lemma}

\begin{proof}
	Since $U$ is faithful, the morphism spaces in $\cD$ are subspaces of the morphism spaces of $\cC$, hence finite dimensional. Since $U$ is a right adjoint, it preserves sub-objects. Thus $U$ sends a decreasing chain of  subobjects to a decreasing chain of subobjects. Since $U$ is exact and faithful, it reflects isomorphisms, and hence $U$ also preserves {\em strictly} decreasing chains of subobjects. Since every such chain in $\cC$ has finite length, the same is true in $\cD$. Let $X \in \cD$ be an object, and let $P \twoheadrightarrow U(X)$ be a surjection in $\cC$ from a projective object. Since $F$ is a left adjoint, it preserves surjections, and since $U$ is faithful (and by Lemma~\ref{lem:fiaghtfulrightadjoint}), the composite
	\begin{equation*}
		F(P) \to FU(X) \to X
	\end{equation*}
	is surjective. Moreover $\cD(F(P), -) \cong \cC(P, U(-))$ is exact, and hence $F(P)$ is projective. Thus $\cD$ also has enough projectives. 
	
Now suppose that $X \in \cD$ is a non-zero object. Since $U$ reflects isomorphisms, $U(X)$ is also non-zero, and hence there exists a non-zero morphism $f: S \to U(X)$ where $S$ is some simple object of $\cC$. The adjoint of this map is the unique map $\overline{f}: F(S) \to X$ such that $f$ factors as $S \to UF(S) \to U(X)$, where the second map is $U(\overline{f})$. Hence, since $f$ is non-zero, $\overline{f}$ must also be non-zero.  Now let $W = \oplus S_i$ be the direct sum of representatives from each of the finitely many isomorphism classes of simple objects of $\cC$. We have shown that for every object $X \in \cD$, there exists a non-zero morphism $F(W) \to X$. If $X$ is simple, then a non-zero morphism is necessarily a surjection. In particular it follows that every simple object of $\cD$ occurs as a simple factor in some composition series for $F(W)$. Since $F(W)$ is finite length and by the Jordan-H\"older theorem, any two composition series have the same simple factors up to permutation and isomorphism, and hence there are finitely many isomorphism classes of simple objects in $\cD$.  
\end{proof}


\begin{proof}[Proof of Prop.~\ref{prop:finitelinearcatsasmodules}]
	Let $A$ be a finite-dimensional $k$-algebra and consider the linear category $\Mod{A}{}$ of finite-dimensional left $A$-modules. The linear category $\Vect_k$ is finite and the  free-forgetful adjunction $A\otimes(-): \Vect_k \leftrightarrows \Mod{A}{}: U$ satisfies the conditions of Lemma~\ref{lem:recognizefinitecats}. Hence the linear category $\Mod{A}{}$ is finite.

	
	
	Now assume that $\cC$ is a finite linear category. Let $\{X_i\}$ be a set of representatives for the isomorphism classes of simples. Let $P_i \to X_i$ be a surjection, with $P_i$ projective, let $P = \oplus P_i$, and let $A = \Hom_{\cC}(P,P)$. As the morphism spaces of $\cC$ are finite-dimensional, $A$ is a finite-dimensional algebra.  (Note that here the algebra structure on $A$ is defined by $a \cdot b := b \circ a$.)
	
We have an adjunction, which we will show is an equivalence:
	\begin{equation*}
		P \otimes_A (-):\Mod{A}{} \leftrightarrows \cC: \Hom_{\cC}(P,-).
	\end{equation*}
	The left adjoint is given by $P \otimes_A M := \coeq\{ P \otimes A \otimes M \rightrightarrows P \otimes M\}$.  In order to show that this adjunction is an equivalence we need only show that the unit and counit maps are isomorphisms.
Because $P$ is a finite length projective module (and hence a summand of a finite rank free module), we have an isomorphism
\begin{equation*}
	\Hom_{\cC}(P, P \otimes_A M) \cong \Hom_{\cC}(P, P) \otimes_A M \cong M.
\end{equation*}
The composition of the unit map $M \rightarrow\Hom_{\cC}(P, P \otimes_A M)$ with this isomorphism is the identity, hence the unit map is an isomorphism. It only remains to show that the counit 
\begin{equation*}
	ev:P \otimes_A \Hom_{\cC}(P, X) \to X
\end{equation*}
is an isomorphism for every $X$. The counit becomes an isomorphism after applying $\Hom_{\cC}(P,-)$, and so the desired result would follow if we knew $\Hom_{\cC}(P,-)$ reflects isomorphisms.

As $P$ is projective, the functor $\Hom_{\cC}(P,-)$ is exact, and so the fact that it reflects isomorphisms is equivalent to that statement that for all $X$, 
\begin{equation*}
	\Hom_{\cC}(P,X) \cong 0 \quad \textrm{ if and only if } \quad X \cong 0.
\end{equation*} 
By construction this holds for all objects $X$ of length at most 1. We prove that it holds for all objects by induction on the length. 

Suppose that $X$ is an object of $\cC$ and, by induction, that for all objects $Y$ with length strictly less than the length of $X$, we know $\Hom_{\cC}(P,Y) \cong 0$ if and only if $Y \cong 0$. By assumption there exists an exact sequence in $\cC$
\begin{equation*}
	0 \to X' \to X \to X'' \to 0
\end{equation*}
with $X''$ simple, and with the length of $X'$ strictly less than the length of $X$. We obtain an exact sequence:
\begin{equation*}
	0 \to \Hom_{\cC}(P,X') \to \Hom_{\cC}(P,X) \to \Hom_{\cC}(P,X'') \to 0.
\end{equation*}
If the middle term is zero, then all terms vanish. By our induction hypothesis, we conclude that $X'' \cong X' \cong 0$, and hence $X$ itself was zero. Thus $\Hom_{\cC}(P,-)$ reflects isomorphisms, as required.
\end{proof}

\subsection{Adjoints and representability of linear functors}

A property of finite linear categories is that they satisfy the following analog of the adjoint functor theorem.  Although this result is well-known, we were unable to find a proof in the literature.  

\begin{proposition} \label{prop:AFT}
	Let $\cC$ and $\cD$ be finite linear categories and let $G: \cC \rightarrow \cD$  be a linear functor. Then the following conditions are equivalent: 
	\begin{enumerate}
		\item $G$ is left exact;  
		\item $G$ is left exact and satisfies the following {\em solution-set condition:} \\  For each $d \in \cD$ there is a finite set $I$ and a collection of arrows $f_i:d \to G(c_i)$ such that every arrow $h:d \to G(c)$ can be written as a composite $h = G(t) \circ f_i$ for some index $i \in I$ and some $t: c_i \to c$; 
		\item $G$ admits a linear left adjoint.
	\end{enumerate}
\end{proposition}
\begin{proof} This is a variation of the proof (in the non-linear setting) given in \cite[V.6.Thm 2]{MR0354798} (see also \cite[Ex. 3-M]{MR0166240}).

($3 \Rightarrow 1$) Suppose that $G$ admits a left adjoint $F$. Then $G$ is itself a right adjoint and hence commutes with all limits. In particular $G$ is a left exact functor. 

($2 \Rightarrow 3$) Note that any left adjoint $F$ to $G$ commutes with coproducts, so is an additive functor.  The adjoint $F$ commutes with colimits and so is also right exact.  Finally the left adjoint $F$ is ${\Vect_k}$-enriched if the natural isomorphism of abelian groups $\hom( F(x), y) \cong \hom(x, G(y))$ is compatible with the scalar multiplication by $k$. That follows by naturality and the fact that $G$ is ${\Vect_k}$-enriched.

To construct a left adjoint for $G$, it suffices (and is necessary) to construct for each $d \in \cD$ a universal arrow $d \to G(c)$, that is an initial object of the comma category $(d \downarrow G)$; the left adjoint may then be constructed pointwise as $F(d) = c$. To this end, fix $d \in \cD$ and suppose that $G$ satisfies the solution-set condition. Define the element $w$ of the comma category as the product of the elements $d \to G(c_i)$. Since $G$ is left exact, it commutes with finite limits, and hence the forgetful functor $(d \downarrow G) \to \cC$ creates finite limits. In particular the comma category has all finite limits.  Thus the product of the elements $d \rightarrow G(c_i)$ exists and is given explicitly by
\begin{equation*}
	w := \bigoplus_i f_i :  d \to G( \oplus_i c_i) = \bigoplus_i G(c_i).
\end{equation*}
The morphism spaces of $(d \downarrow G)$ are finite-dimensional vector spaces, and so we may choose a finite basis for $\Hom(w,w)$. Let $v$ be defined as the equalizer of this finite collection of maps together with the zero map. Again, this finite limit exists and may be created in $\cC$. Note that $v$ consequently equalizes the zero map and any linear combination of the basis maps. Thus $v$ equalizes {\em all} endomorphisms of $w$, and hence $v$ is initial; see  \cite[V.6.Thm 1]{MR0354798}. 

($1 \Rightarrow 2$) Finally, suppose that $G$ is left exact. Fix an object $d \in \cD$ and define a class $S_d$ of objects of $(d \downarrow G)$ as follows. An object $(c, f:d \to G c)$ belongs to $S_d$ if and only if for any sub-object $c' \subseteq c$ and factorization $d \to G c' \to G c$ it follows that either $c' \cong 0$ or the inclusion is an isomorphism $c' \cong c$.  The class $S_d$ consists of, in the language of \cite[Ex. 3-J]{MR0166240}, those objects $(c,f)$ that are {\em generated} by $d$. If there is a finite set $I$ of isomorphism classes of objects of $S_d$ then a choice of representatives $\{(c_i, f_i)\}_{i \in I}$ forms a collection satisfying the solution-set condition.

We must show that the class $S_d$ has finitely many isomorphism classes of objects. Let $q \in \cC$ denote the direct sum of representatives of the isomorphism classes of simple objects in $\cC$. Choose a basis $\{e_j\}$ for the vector space $\hom_{\cD}(d, G q)$, and consider the object $x := \oplus_j( q, e_j) \bigoplus (q,0) \in (d \downarrow G)$. We will show that every object of $S_d$ is isomorphic to a subobject of $x$, and hence there are only a finite number of isomorphism classes of objects of $S_d$. Note that because every object of $\cC$ has finite length, $q$ is a {\em cogenerator} of $\cC$, in the sense that the functor $\hom(-, q)$ is faithful. Let $(c, f)$ be an element of $S_d$. If $f: d \to G c$ is the zero map then $c$ is necessarily simple in $\cC$. In this case $(c, f)$ is a subobject of $(q, 0)$, and hence of $x$. 

Thus we may assume, without loss of generality, that  $(c, f)$ is an element of $S_d$ in which $f$ is not the zero map. In this case consider the map of vector spaces
\begin{equation*}
	\hom_{\cC}(c,q) \stackrel{G}{\to} \hom_{\cD}(G c, G q) \stackrel{f^*}{\to} \hom_{\cD}(d, G q).
\end{equation*} 
If $g: c \to q$ is in the kernel of the above composite map, then, by the left exactness of $G$, the kernel $\ker g$ is a subobject of $c$ admitting a factorization $d \to G(\ker g) \to G c$. Hence, by the defining property of $S_d$, we have either $\ker g = 0$ or $\ker g = c$. The latter case only occurs when $g$ is the zero map. In the former case, in which $g$ is injective, we have that $G(g)$ is also injective. Thus since $G(g) \circ f = 0$, it follows that $f = 0$, a case we have ruled out by assumption. We conclude that the map $f^* \circ G$ is an inclusion. Choose a basis  $\{ \epsilon_i \}$ for $\hom_{\cC}(c,q)$, and let $(a_{ij})$ be the matrix coefficients for the map $f^* \circ G$ in the bases $\{ \epsilon_i \}$ and $\{ e_j \}$. Since the map $f^* \circ G$ is an inclusion and $q$ is a cogenerator, the natural map
\begin{equation*}
	c \stackrel{\oplus \epsilon_i}{\to} \oplus_i q \stackrel{(a_{ij}) }{\to} \oplus_j q
\end{equation*} 
is a monomorphism in $\cC$. This monomorphism exhibits $(c,f)$ as a subobject of $x = \oplus_j( q, e_j) \bigoplus (q,0)$. 
\end{proof}

\begin{remark}
	The statement of the above proposition assumes that $\cC$ and $\cD$ are finite linear categories as this is the only case we will use. However the above proof reveals that the proposition holds under weaker assumptions. It is not necessary for either $\cC$ or $\cD$ to have enough projectives, and moreover we do not need $\cD$ to have finitely many isomorphisms classes of simple objects, nor to have only objects of finite length. As written we do need $\cD$ to be enriched in finite-dimensional vector spaces. Of course other variations of the proposition are possible. 
\end{remark}

\begin{corollary}
	A right exact linear functor between finite linear categories always admits a right adjoint (which is left exact but may not be right exact). A left exact linear functor between finite linear categories always admits a left adjoint (which is right exact but may not be left exact). 
\end{corollary}

\begin{proof}
	The second statement is just a portion of the above proposition; the first follows by passing to the opposite linear category.  
\end{proof}

\begin{corollary} \label{cor:representable}
If $\cC$ is a finite linear category, then a functor $G: \cC^\textrm{op} \to \Vect_k$ is representable, that is $G(-) \cong \Hom_{\cC}(-,x)$ for some $x \in \cC$, if and only if $G$ is left exact. 
\end{corollary}

\begin{proof}
	The represented functor $\Hom_{\cC}(-, x):\cC^\text{op} \to \Vect_k$ sends all limits in $\cC^\text{op}$ (i.e. colimits in $\cC$) to limits in $\Vect_k$. In particular it is left exact. 
%
%
Conversely, if $G: \cC^\text{op} \to \Vect_k$ is left exact, then by the above proposition it admits a left adjoint $F$. Thus for every every object $c \in \cC$ we have a natural isomorphism
	\begin{equation*}
		G(c) \cong \Hom_{\Vect_k}( k, G(c)) \cong \Hom_{\cC^\op}( F(k), c) = \Hom_{\cC}(c, F(k) ).
	\end{equation*}
	In other words,  the object $F(k)$ represents the functor $G$. 
\end{proof}

\subsection{Tensor categories}


Let $\cC$ be monoidal category. $\cC^\mp$ will denote its {\em monoidal opposite}; this has the same underlying category as $\cC$, but $x \otimes^{\cC^\mp} y : = y \otimes^{\cC} x$. 

\begin{definition} \label{def:rigid}
	An object $x \in \cC$ admits a {\em right dual} if there exists an object $x^*$ and morphisms, the {\em coevaluation} $\eta: 1 \to x \otimes x^*$ and the {\em evaluation} $\varepsilon: x^* \otimes x \to 1$, satisfying the following pair of `zigzag' equations:
	\begin{align*}
		(\id_{x} \otimes \varepsilon  ) \circ (  \eta \otimes \id_{x}) &= \id_{x} \\
		(\varepsilon \otimes \id_{x^*}) \circ (\id_{x^*} \otimes \eta) &= \id_{x^*};
	\end{align*}
	when these equations are satisfied, the object $x$ is called a left dual of the object $x^*$.  A monoidal category $\cC$ is {\em rigid} if each object of $\cC$ and each object of $\cC^\mp$ admit right duals, in other words if each object of $\cC$ admits both a right and a left dual. 
\end{definition}

\begin{definition}
	A {\em linear monoidal category} is a monoidal category $\cC$ such that $\cC$ is a linear category and the functor $\otimes$ is bilinear.  A {\em tensor category} is a rigid linear monoidal category.  A {\em finite tensor category} is a finite rigid linear monoidal category.
\end{definition}

\nid Here by bilinear, or more generally multilinear, we mean the following.  If $\{\cM_\alpha\}$ denotes a collection of linear categories then a {\em multilinear functor} from $\{\cM_\alpha\}$ into a linear category $\cN$ is a functor $F: \prod \cM_\alpha \to \cN$ such that $F$ is linear in each variable separately.


\begin{example}
	The category $\Vect_k$ is a finite tensor category. When $A$ is an algebra in $\Vect_k$, the categories of finite-dimensional left and right modules, denoted $\Mod{A}{}$ and $\Mod{}{A}$, are finite linear categories. More generally, when $\cC$ is a finite tensor category and $A$ is an algebra object in $\cC$ (also known as a monoid object), the categories $\Mod{A}{}(\cC)$ and $\Mod{}{A}(\cC)$ of left and right $A$-module objects in $\cC$ are also finite linear categories.
\end{example}

\begin{example}
When $K$ is a finite group, the category of finite-dimensional $K$-graded vector spaces $\Vect[K]$ is a finite tensor category.  Again when $K$ is finite, the category of finite-dimensional representations $\Rep(K)$ is a finite tensor category.  Similarly, if $H$ is a finite-dimensional Hopf algebra, then the category $\Rep(H)$ is a finite tensor category.
\end{example}


\begin{lemma} \cite[2.1.8]{MR1797619} \cite[Prop. 4.2.1]{egno-book}
\label{lma:RigidIsExact}
	Let $(\cC, \otimes)$ be a finite tensor category. The bilinear functor $\otimes: \cC \times \cC \to \cC$ is exact in both variables. 
\end{lemma}

\begin{proof}
	The units and counits give rise to natural isomorphisms
 \begin{equation*} 
 	\Hom(x \otimes y, z) \cong \Hom( x, z \otimes y^*) \cong \Hom(y, {}^*x \otimes z).
 \end{equation*}
	Hence for all $x$ and $y$ the functors $(-)\otimes x$ and $y \otimes (-)$ admit both left and right adjoints, and are consequently exact. 
\end{proof}


\section{Module categories} \label{sec:tc-bimod}

\subsection{Module categories, functors, and transformations} 

A module category is a linear category with an action by a linear monoidal category, and a bimodule category is a linear category with two commuting actions by linear monoidal categories:
\begin{definition}
Let $\cC$ and $\cD$ be linear monoidal categories.
A \emph{left $\cC$-module category} is a linear category $\cM$ together with a bilinear functor $\otimes^{\cM}: \cC \times \cM \to \cM$ and natural isomorphisms
\begin{align*}
		\alpha: & \;    \otimes^{\cM} \circ \; (\otimes^{\cC} \times id_{\cM}) \cong  \otimes^{\cM} \circ (id_{\cC} \times \otimes^{\cM}), \\
		\lambda: & \; \otimes^{\cM} (1_{\cC} \times -) \cong id_{\cM},
\end{align*}
satisfying the evident pentagon identity and triangle identities.  We will use the notation $c \otimes m := \otimes^\cM(c \times m)$.  A \emph{right $\cD$-module category} is defined similarly.  A \emph{$\cC$--$\cD$-bimodule category} is a linear category $\cM$ with the structure of a left $\cC$-module category and the structure of a right $\cD$-module category, together with a natural associator isomorphism $(c \otimes m) \otimes d \cong c \otimes (m \otimes d)$ satisfying two additional pentagon axioms and two additional triangle axioms.  
\end{definition}
\nid By a finite module or bimodule category we will mean simply a module or bimodule category whose underlying linear category is finite. 

\begin{example}
	Every linear category is a $\Vect_k$--$\Vect_k$-bimodule category in an essentially unique way: any two such structures are naturally isomorphic via a unique natural isomorphism.
\end{example}

\begin{example} \label{ex:ModulesAreModules}
	For any algebra object $A$ in a tensor category $\cC$, the category $\Mod{}{A}(\cC)$ of right $A$-modules in $\cC$ is a left $\cC$-module category.  Similarly, the category $\Mod{A}{}(\cC)$ of left $A$-modules in $\cC$ is a right $\cC$-module category.
\end{example}

\begin{example}
For a tensor category $\cC$, the category $\cC$ itself is naturally a $\cC$--$\cC$-bimodule category with the actions given by the tensor product and the natural transformations by the associator.
\end{example}

\begin{example}
Let $K$ be a finite group. 
$\Vect$ is a right $\Rep(K)$-module category with the action given by forgetting and then tensoring.  Similarly $\Vect$ is a left $\Vect[K]$-module category.  Note that $\Vect$ can be given the structure of a $\Rep(K)$--$\Vect[K]$-bimodule category in several ways.  For instance, one could choose the associator $V = (V \otimes k) \otimes k_g \rightarrow V \otimes (k \otimes k_g) = V$ to be the identity, or to be given by the action of $g$.
\end{example}

The notions of functor between module categories and transformation of such functors are as expected:
\begin{definition}
A linear functor $\cF: \cM \ra \cN$ from the $\cC$-module category $\cM$ to the $\cC$-module category $\cN$ is a:
\begin{enumerate}
	\item \emph{lax left $\cC$-module functor} if it is equipped with a natural transformation 
	\begin{equation*}
		\overline{f}_{c,m}:c \otimes \cF(m) \rightarrow \cF(c \otimes m)
	\end{equation*}
	 satisfying the evident pentagon relation and triangle relation.
	\item \emph{oplax left $\cC$-module functor} if it is equipped with a natural transformation 
	\begin{equation*}
		f_{c,m}:\cF(c \otimes m) \rightarrow c \otimes \cF(m)
	\end{equation*}
	satisfying the evident pentagon relation and triangle relation.
	\item \emph{strong left $\cC$-module functor} if it is equipped with a natural isomorphism 
	\begin{equation*}
		f_{c,m}:\cF(c \otimes m) \rightarrow c \otimes \cF(m)
	\end{equation*}
	satisfying the evident pentagon relation and triangle relation.
\end{enumerate} 
Lax, oplax, and strong right module functors and a bimodule functors are defined similarly. By convention all module functors will be assumed to be strong unless stated otherwise. 
\end{definition}

\begin{definition}
A \emph{left $\cC$-module transformation} from the $\cC$-module functor $\cF$ to the $\cC$-module functor $\cG$ is a natural transformation $\eta: \cF \ra \cG$ satisfying the condition $(\id_c \otimes \eta_m) \circ f_{c,m} = g_{c,m} \circ \eta_{c \otimes m}$.  Right module transformations and bimodule transformations are defined similarly.
\end{definition}

\nid Note that bimodule categories, functors, and transformations form a strict 2-category.  

\begin{example}
Any linear functor from $\Vect$ to $\Vect$ is determined by a choice of finite-dimensional vector space $F(k)$.  For $K$ a finite group, a $\Vect[K]$-module functor from $\Vect$ to $\Vect$ consists of a vector space $V$ together with for each $g$ a map $\rho(g): V \rightarrow V$ satisfying $\rho(gh) = \rho(g)\rho(h)$.  Giving a natural transformation between two such module functors $(V, \rho)$ is the same as giving a map of representations.  Thus $\Fun_{\Vect[K]}(\Vect, \Vect) \cong \Rep(K)$.
\end{example}

\begin{example}
	Let $\cC$ be a finite tensor category and $\cM$ and $\cN$ finite left $\cC$-module categories. The categories $\Fun_{\cC}(\cM, \cM)$ and $\Fun_{\cC}(\cN, \cN)$ of right exact $\cC$-module endofunctors are finite linear monoidal categories. 
	The linear category $\Fun_{\cC}(\cM, \cN)$ is a $\Fun_{\cC}(\cM, \cM)$--$\Fun_{\cC}(\cN, \cN)$-bimodule category (where the actions are given by $F \otimes G := G \circ F$). 
\end{example}

The following result is well known \cite[Rmk 4]{MR1976459}, but we were unable to find a proof in the literature.

\begin{lemma} \label{lem:laxisstrong}
	Let $\cC$ be a tensor category. Then every lax (or oplax) $\cC$-module functor is strong.  
\end{lemma} 

\begin{proof}
We show the lax case; the oplax case is similar.  Suppose that $F: \cM \ra \cN$ is a linear functor between $\cC$-module categories, and $f_{c,m}:  c \otimes F(m) \rightarrow F(c \otimes m)$ is a natural transformation making $F$ into a lax module functor.  
The inverse to this natural transformation is given explicitly by the mate of $f_{c^*,m}$, namely 
$$F(c \otimes m) \rightarrow c \otimes c^* \otimes F(c \otimes m) \rightarrow c \otimes F(c^* \otimes c \otimes m) \rightarrow c \otimes F(m)$$
Here the first map is given by the coevaluation, the second map is $f_{c^*,m}$, and the third map is evaluation.  Diagrammatically:
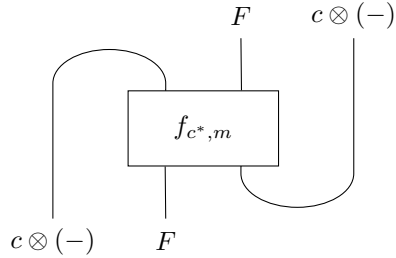
\begin{center}
\begin{tikzpicture}[yscale=0.6]
	\node (A) at (2,2) [minimum height=1cm,minimum width=2cm, draw] {$f_{c^*,m}$};
	\draw (0,0) -- (0,3) arc (180:0:0.75cm) |- (A.135);
	\draw (A.225) -- (1.5,0);
	\draw (A.45) -- (2.5,4);
	\draw (A.315) -| (2.5,1) arc (180:360:0.75cm) -- (4,4);
	\node at (2.5,4.5) {$F$};
	\node at (4,4.5) {$c \otimes (-)$};
	\node at (0,-0.5) {$c \otimes (-)$};
	\node at (1.5, -0.5) {$F$};
\end{tikzpicture}
\end{center} 
We need to check that this map is inverse to $f_{c^*,m}$.  That this map is inverse to $f_{c^*,m}$ follows from the associativity condition for module functors and naturality, as illustrated here:
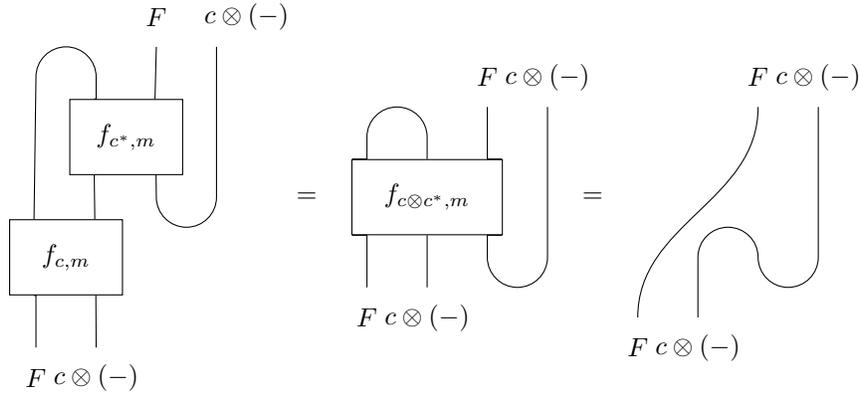
\begin{center}
\begin{tikzpicture}[yscale=0.8, xscale=0.8, smooth]
	\node (A) at (2,2) [minimum height=1cm,minimum width=1.5cm, draw] {$f_{c^*,m}$};
	\node (B) at (1,0) [minimum height=1cm,minimum width=1.5cm, draw] {$f_{c,m}$};
	\draw (B.130) -- (.5,3) arc (180:0:0.5cm) |- (A.130) ;
	\draw (A.230) -- (B.53);
	\draw (A.53) -- (2.5,3.5);
	\draw (A.308) -- (2.5,1) arc (180:360:0.5cm) -- (3.5,3.5);
	\draw (B.308) -- (1.5, -1.5);
	\draw (B.230) -| (.5, -1.5);
	\node at (2.5,4) {$F$};
	\node at (4,4) {$c \otimes (-)$};
	\node at (.5, -2) {$F$};
	\node at (1.5, -2) {$c \otimes (-)$};
	
	\node at (5, 1) {$=$};
	
	\begin{scope}[xshift=5cm, yshift = -1cm]
		\node (A) at (2,2) [minimum height=1cm,minimum width=2cm, draw] {$f_{c \otimes c^*,m}$};
		\draw (A.153) -| (1,3) arc (180:0:0.5cm) -- (A.90) ;
		\draw (A.27) -| (3,3.5);
		\draw (A.333) -| (3,1) arc (180:360:0.5cm) -- (4, 3.5);
		\draw (A. 270) -- (2,.5);
		\draw (A.207) -| (1,.5);
		\node at (3,4) {$F$};
		\node at (4,4) {$c \otimes (-)$};
		\node at (1,0) {$F$};
		\node at (2,0) {$c \otimes (-)$};
	\end{scope}
	
	\node at (9.75, 1) {$=$};
	
	\begin{scope}[xshift=10.5cm, yshift=-1cm]
		\draw (0, 0) to [out=90, in=270] (2,3.5);
		\draw (1,0) -- (1,1) arc (180:0:.5cm) arc (180:360:.5cm) -- (3,3.5);
		\node at (2,4) {$F$};
		\node at (3,4) {$c \otimes (-)$};
		\node at (0,-.5) {$F$};
		\node at (1,-.5) {$c \otimes (-)$};
	\end{scope}
\end{tikzpicture} 
\end{center}
\end{proof}

\begin{lemma}\label{lem:module-adjoint-main}
	Let $\cC$ and $\cD$ be linear monoidal categories. Let  $\cM$ and  $\cN$  be  $\cC$--$\cD$-bimodule categories, and let $F: \cM \to \cN$ be an oplax (respectively lax) $\cC$--$\cD$-bimodule functor. If the underlying functor of $F$ has a right (respectively left) adjoint as a functor, then this adjoint admits the structure of a lax (respectively oplax) $\cC$--$\cD$-bimodule functor such that the following four squares commute
	\begin{center}
	\begin{tikzpicture}
			\node (LT) at (0, 1.5) {$F(c \otimes G(x))$};
			\node (LB) at (0, 0) {$FG(c \otimes x)$};
			\node (RT) at (3, 1.5) {$c \otimes FG(x)$};
			\node (RB) at (3, 0) {$c \otimes x$};
			\draw [->] (LT) -- node [left] {$\overline{g}_{c,x}$} (LB);
			\draw [->] (LT) -- node [above] {$f_{c, G(x)}$} (RT);
			\draw [->] (RT) -- node [right] {$\varepsilon$} (RB);
			\draw [->] (LB) -- node [below] {$\varepsilon$} (RB);
	\end{tikzpicture}
	\begin{tikzpicture}
			\node (LT) at (0, 1.5) {$F(G(x) \otimes d)$};
			\node (LB) at (0, 0) {$FG(x \otimes d)$};
			\node (RT) at (3, 1.5) {$FG(x) \otimes d$};
			\node (RB) at (3, 0) {$x \otimes d$};
			\draw [->] (LT) -- node [left] {$\overline{g}_{x,d}$} (LB);
			\draw [->] (LT) -- node [above] {$f_{G(x), d}$} (RT);
			\draw [->] (RT) -- node [right] {$\varepsilon$} (RB);
			\draw [->] (LB) -- node [below] {$\varepsilon$} (RB);
	\end{tikzpicture}

	\begin{tikzpicture}
			\node (LT) at (0, 1.5) {$c \otimes y$};
			\node (LB) at (0, 0) {$c \otimes GF(y)$};
			\node (RT) at (3, 1.5) {$GF(c \otimes y)$};
			\node (RB) at (3, 0) {$G( c \otimes F(y))$};
			\draw [->] (LT) -- node [left] {$\eta$} (LB);
			\draw [->] (LT) -- node [above] {$\eta$} (RT);
			\draw [->] (RT) -- node [right] {$f_{c,y}$} (RB);
			\draw [->] (LB) -- node [below] {$\overline{g}_{c, F(y)}$} (RB);
	\end{tikzpicture}
	\begin{tikzpicture}
			\node (LT) at (0, 1.5) {$y \otimes d$};
			\node (LB) at (0, 0) {$GF(y) \otimes d$};
			\node (RT) at (3, 1.5) {$GF(y \otimes d)$};
			\node (RB) at (3, 0) {$G( F(y) \otimes d)$};
			\draw [->] (LT) -- node [left] {$\eta$} (LB);
			\draw [->] (LT) -- node [above] {$\eta$} (RT);
			\draw [->] (RT) -- node [right] {$f_{y, d}$} (RB);
			\draw [->] (LB) -- node [below] {$\overline{g}_{ F(y),d}$} (RB);
	\end{tikzpicture}
	\end{center}
	(similar squares commute in the left adjoint case). 
	
	Moreover if $F$ is a strong $\cC$--$\cD$-bimodule functor and the unit and counit of the adjunction are isomorphisms, then the adjoint is also a strong $\cC$--$\cD$-bimodule functor, and the unit and counit are bimodule transformations.   	
\end{lemma}

\begin{proof}
Suppose that $G$ is the right adjoint to the underlying functor of $F$; we will show that $G$ naturally has the structure of a lax $\cC$--$\cD$-bimodule functor.  The result for left adjoints is similar.

The natural transformation $\overline{g}_{c,n}: c \otimes G(n) \rightarrow G(c \otimes n)$ is given by the mate:
$$c \otimes G(n) \rightarrow G F(c \otimes G(n)) \rightarrow G(c \otimes FG(n)) \rightarrow G(c \otimes n);$$
here the first map is the unit of the adjunction, the second map is the natural transformation given by the module functor structure on $F$, and the third map is the counit.  Diagrammatically: 
\begin{center}
\begin{tikzpicture}[yscale=0.6]
	\node (A) at (2,2) [minimum height=1cm,minimum width=2cm, draw] {$f_{c,G(-)}$};
	\draw (0,0) -- (0,3) arc (180:0:0.75cm) |- (A.135);
	\draw (A.225) -- (1.5,0);
	\draw (A.45) -- (2.5,4);
	\draw (A.315) -- (2.5,1) arc (180:360:0.75cm) -- (4,4);
	\node at (2.5,4.5) {$c \otimes (-)$};
	\node at (4,4.5) {$G$};
	\node at (0,-0.5) {$G$};
	\node at (1.5, -0.5) {$c \otimes (-)$};
\end{tikzpicture}
\end{center}

\noindent 
Providing the structure of a $\cD$-module functor is similar, and these left and right module functor structures are compatible---that is they form a bimodule functor structure. The four commuting squares are easily verified. 

If the unit, counit, and transformations $f$ are invertible, then the above transformation will also be invertible, and hence $G$ will be strong, and the four commuting squares are equivalent to the statement that the unit and counit are bimodule transformations.  
\end{proof}

This yields the following characterization of equivalences of bimodule categories.

\begin{corollary}\label{cor:Recog_equiv_of_bimod}
	Given two $\cC$--$\cD$-bimodule categories $\cM$ and $\cN$, and a bimodule functor $F:\cM \to \cN$, then $F$ is an equivalence of $\cC$--$\cD$-bimodule categories if and only if it induces an equivalence of underlying categories.
\end{corollary}

%

The following corollary was left as an exercise to the reader in \cite[\S 3.3]{EO-ftc}.
\begin{corollary} \label{cor:module-adjoint}
Let $\cC$ and $\cD$ be tensor categories. Let  $\cM$ and  $\cN$  be  $\cC$--$\cD$-bimodule categories, and let $F: \cM \to \cN$ be a $\cC$--$\cD$-bimodule functor.  If the underlying functor of $F$ has a right (respectively left) adjoint as a functor, then $F$ has a right (respectively left) adjoint $\cC$--$\cD$-bimodule functor such that the unit and counit maps are bimodule natural transformations.
\end{corollary}

\begin{proof}
	By Lemma~\ref{lem:module-adjoint-main} the right adjoint $G$ has the structure of a lax $\cC$--$\cD$-bimodule functor. By Lemma~\ref{lem:laxisstrong} it is actually strong. In this case the four commuting squares of Lemma~\ref{lem:module-adjoint-main} imply that the unit and counit are bimodule natural transformations. The case of left adjoints is similar. 
\end{proof}


\begin{lemma}\label{lem:partially_exact_action} \cite[Ex 7.3.2]{egno-book} \cite[Prop. 2.1.8]{MR1797619}
	Let $\cC$ be a tensor category and let $\cM$ be a $\cC$-module category. Then for each object $c \in \cC$, the action map $c \otimes (-): \cM \to \cM$ is exact. 
\end{lemma}

\begin{proof}
	For each $c \in \cC$, the functor $c \otimes (-)$ admits both left and right adjoints, namely $c^* \otimes (-)$ and ${}^*c \otimes (-)$ respectively. 
\end{proof}

\subsection{Module categories are categories of modules}

Just as any finite linear category is a category of modules over an algebra, any finite module category over a finite tensor category is a category of module objects over an algebra object.  This result is one of the main theorems of \cite{egno-book}, and is essential to the structure theory of finite tensor categories.  The key construction underlying the proof is Ostrik's notion of the enriched hom for module categories \cite{MR1976459}.  

The following proposition is an elaboration of results of \cite{MR1976459} and \cite{EO-ftc}. 
\begin{proposition} \label{thm:enrichment-of-mod-cats}
	Let $\cC$ be a finite linear monoidal category and let $\cM$ be a finite $\cC$-module category. Assume that the action map $\cC \times \cM \to \cM$ is right exact in each variable. 
		Then $\cM$ is enriched, tensored, and cotensored over $\cC$, with the tensor structure given by its structure as a $\cC$-module category. 
\end{proposition}

\begin{proof}
	We need functorial assignments of objects $\IHom(m', m) \in \cC$ and $m^c \in \cM$ for every $m,m' \in \cM$ and $c \in \cC$, and natural isomorphisms
	\begin{equation*}
		\Hom_\cC(c, \IHom(m', m)) \cong \Hom_\cM(c \otimes m', m) \cong \Hom_\cM( m', m^c)
	\end{equation*}
	witnessing adjunctions
	\begin{equation*}
			c \otimes - \quad \dashv\quad -^c \quad \textrm{ and } \quad - \otimes m' \quad \dashv \quad \IHom(m', -).
	\end{equation*}
Such assignments exist precisely if the functors
\begin{align*}
	\Hom_\cM(-\otimes m', m): \; & \cC^\op \to \Vect \\
	\Hom_\cM(c \otimes -, m): \; & \cM^\op \to \Vect
\end{align*}
are representable. The representing objects will be $\IHom(m', m)$ and $m^c$, respectively. Both of these functors are left exact, and hence by Corollary \ref{cor:representable} both of these functors are indeed representable.  
\end{proof}

\begin{remark} \label{rem-enrich}
	The existence of the enrichment $\IHom(m', m)$ only requires $\cC$ to be finite and the action to be right exact in $\cC$, and similarly the existence of the cotensor $m^c$ only requires $\cM$ to be finite and the action to be right exact in $\cM$.
\end{remark}

\begin{example}
Consider $\Vect$ as a module category over $\Vect[K]$.  Then $\IHom(1,1)$ is $k[K]$ with each $g \in K$ having grading $g$.
\end{example}

The goal of this subsection is to prove a generalization of the following result.

\begin{theorem}{\cite[Cor. 7.10.5]{egno-book}, \cite[Thm 1]{MR1976459}} \label{thm:EGNO2.11.6} 
	Let $\cM$ be a left module category over a finite tensor category $\cC$, and assume the action is right exact in $\cC$. If $\cM$ is finite as a linear category, then there exists an algebra object $A \in \cC$ together with an equivalence $\cM \simeq \Mod{}{A} (\cC)$ as left $\cC$-module categories. 
\end{theorem}

\begin{example} \label{eg:vect}
The $\Vect[K]$-module category $\Vect$ is equivalent to the category of modules over the algebra object $k[K] \in \Vect[K]$; here as before the element $g \in K$ has grading $g$.  
\end{example}

\begin{example} \label{ex:lax-module}
	In this theorem, it is necessary to assume $\cC$ is rigid.  Let $\cR \cong \Vect \oplus (\Vect \cdot X)$ be the non-rigid linear monoidal category consisting of pairs of vector spaces, which we write as $V_1 + V_2 X$, with tensor product given by 
	\begin{equation*}
		(V_1 + V_2 X) \otimes (W_1 + W_2 X) = V_1 \otimes W_1  +  (V_1 \otimes W_2 \oplus V_2 \otimes W_1)X.
	\end{equation*} 
	Up to equivalence there is a unique choice of associators and unitors making this a linear monoidal category. 
This is a categorification of the ring $k[x]/(x^2)$.  It is both finite and semisimple, but it is not rigid: the object $X$ cannot have a dual as there is no object $Z$ such that $Z \otimes X$ has a non-zero map to or from the unit object. 
	
	There is a tensor functor $F:\cR \to \Vect$ given by $(V_1 + V_2 X) \mapsto V_1$. This gives the category $\Vect$ the structure of a (left) $\cR$-module category, and moreover $F$ is naturally an $\cR$-module map. $F$ has both a left and a right adjoint, which agree and are given by the functor $G: \Vect \to \cR$ sending $W \in \Vect$ to $(W + 0 X) \in \cR$. 
	It is not possible to give $G$ the structure of a (strong) $\cR$-module functor (in contrast to Cor.~\ref{cor:module-adjoint}). Moreover there is no algebra object $A \in \cR$ such that $\Vect$ is equivalent to $\Mod{}{A}(\cR)$ as linear categories, let alone as $\cR$-module categories. 
\end{example}

In order to see how rigidity appears in the proof of the above theorem, we will first show a more general result that does not assume rigidity.

\begin{definition}
	Let $\cC$ be a finite linear monoidal category and let $\cM$ be a finite $\cC$-module category. Assume that the action is right exact in $\cC$, so $\cM$ is enriched over $\cC$ by Remark~\ref{rem-enrich}. 
	An object $p \in \cM$ will be called {\em $\cC$-projective} if $\IHom(p, -)$ is right exact (it is automatically left exact). An object $p$ will be called a {\em $\cC$-generator} if $\IHom(p,-)$ is faithful.
\end{definition}

\begin{lemma}\label{lem:gen}
	Let $\cC$ be a finite linear monoidal category and let $\cM$ be a finite $\cC$-module category in which the action is right exact in $\cC$. Then for an object $p \in \cM$ the following are equivalent:
	\begin{enumerate}
		\item for each object $x \in \cM$ the canonical map $\IHom(p,x) \otimes p \twoheadrightarrow x$ is a surjection;
		\item for each object $x \in \cM$ there exists an object $c \in \cC$ and a surjection $c \otimes p \twoheadrightarrow x$;
		\item $p \in \cM$ is a $\cC$-generator;
	\end{enumerate}
	In particular an ordinary generator (i.e. an object such that $\Hom_\cM(p,-)$ is faithful) is also a $\cC$-generator.
\end{lemma}
\begin{proof}
	Clearly $(1) \Rightarrow (2)$. Suppose that $(2)$ holds and let $f: x \to y$ be a map such that $\IHom(p, x) \to \IHom(p,y)$ is zero. It follows that for all $c$ and all $c \to \IHom(p,x)$, the induced composite $c \otimes p \to x \to y$ is zero. By choosing $c$ such that $c \otimes p \to x$ is surjective, it follows that $f=0$. Hence $(2) \Rightarrow (3)$. Finally $(3) \Rightarrow (1)$ by Lemma~\ref{lem:fiaghtfulrightadjoint}.
\end{proof}


\begin{example} \label{ex:rigid_all_C-proj}
	Let $\cC$ be a finite tensor category. We may view $\cC$ as a left $\cC$-module category over itself. Since $\cC$ is rigid, we have isomorphisms $\IHom(x,y) \cong y \otimes x^*$ and $(y^x) \cong {}^*x \otimes y$ (cf the proof of Lemma \ref{lma:RigidIsExact}). Moreover by Lemma \ref{lma:RigidIsExact}, in this case the tensor product is exact in each variable, hence every object of $\cC$ is $\cC$-projective, even those objects that are not projective in the usual sense. 
\end{example}

\begin{theorem} \label{thm:C-module-Embedding} 

	Let $\cM$ be a left module category over a finite (not necessarily rigid) linear monoidal category $\cC$. Assume that the action is right exact in $\cC$. Fix an object $p \in \cM$, and set $A = \IHom(p,p) \in \cC$. The object $A$ is naturally an algebra object in $\cC$, and there is a $\cC$-module functor:
	\begin{align*}
		F:   \Mod{}{A}(\cC) \to \cM \\
		F(x_A) := \coeq \left( x \otimes \IHom(p,p) \otimes p \rightrightarrows x \otimes p \right).
	\end{align*}
The functor $\IHom(p,-): \cM \to \Mod{}{A}(\cC)$ is right adjoint to $F$. 

Moreover, assume that $\IHom(p,-)$ may be equipped with the structure of a $\cC$-module functor and that the unit and counit of the adjunction $F \dashv \IHom(p,-)$ are morphisms of $\cC$-module functors. Then the $\cC$-module adjunction  $F \dashv \IHom(p,-)$ induces an equivalence of left $\cC$-module categories $\cM \simeq \Mod{}{A}(\cC)$ if and only if  $p$ is a $\cC$-projective $\cC$-generator.
\end{theorem}

\noindent The proofs given in \cite{egno-book} and \cite{MR1976459} at first appear to depend on  the rigidity of $\cC$. Indeed the first step of the proof of \cite[Thm 7.10.1]{egno-book}
invokes 
\cite[Lemma 7.9.4]{egno-book}
whose proof makes explicit use of rigidity. The same lemma is used in step (3) of the proof of \cite[Thm 1]{MR1976459}. However with a little care the rigidity assumption can be avoided (cf. \cite[\S 7.24, Note 7.10]{egno-book}). 

\begin{proof}[Proof of Thm.~\ref{thm:C-module-Embedding}]
	We have a series of natural isomorphisms:
	\begin{align*}
		\Hom_{\Mod{}{A}(\cC)}&(b, \IHom(p,x))  \cong \Hom_{\Mod{}{A}(\cC)}( \coeq \left( b \otimes A \otimes A_A \rightrightarrows b \otimes A_A  \right), \IHom(p,x)) \\
		& \cong \eq \left( \Hom_{\Mod{}{A}(\cC)}(b \otimes A_A, \IHom(p,x) )  \rightrightarrows \Hom_{\Mod{}{A}(\cC)}(  b \otimes A \otimes A_A, \IHom(p,x))  \right) \\
		& \cong \eq \left( \Hom_{\cC}(b, \IHom(p,x) )  \rightrightarrows \Hom_{\cC}(  b \otimes A, \IHom(p,x))  \right) \\
		& \cong \eq \left( \Hom_{\cM}(b \otimes p, x )  \rightrightarrows \Hom_{\cM}(  (b \otimes A) \otimes p, x)  \right) \\
		& \cong  \Hom_{\cM}( \coeq \left( (b \otimes \IHom(p,p)) \otimes p \rightrightarrows b \otimes p \right), x) \\
		& \cong \Hom_{\cM}( F(b), x)
	\end{align*}  
	where $b \in \Mod{}{A}(\cC)$ and $x \in \cM$. This establishes that $\IHom(p,-): \cM \to \Mod{}{A}(\cC)$ is indeed right adjoint to $F$.

For the second part of the theorem first observe that if $F \dashv \IHom(p,-)$ induces an equivalence of left $\cC$-module categories $\cM \simeq \Mod{}{A}(\cC)$, then the functor also denoted $\IHom(p, -): \cM \to \cC$ is faithful and exact. Hence $p$ is by definition a $\cC$-projective $\cC$-generator.

Now suppose that $p$ is a $\cC$-projective $\cC$-generator.  We will proceed similarly to Prop.~\ref{prop:finitelinearcatsasmodules}.
We wish to show that the unit and counit of the adjunction $F \dashv \IHom(p,-)$ are isomorphisms. By assumption, $\IHom(p, -)$ is a $\cC$-module functor and hence for each $c \in \cC$ we have a natural isomorphism
\begin{equation*}
	\IHom(p, c \otimes p) \cong c \otimes \IHom(p, p).
\end{equation*}
 Since $\IHom(p, -)$ is exact it commutes with finite colimits, such as coequalizers. Hence for each $A$-module $x_A \in \Mod{}{A}(\cC)$ we have
 \begin{equation*}
 	\IHom(p, F(x_A)) \cong x \otimes_A \IHom(p,p) \cong x.
 \end{equation*}
Composing with the unit gives the identity map of $x$, hence the unit is an isomorphism. 

Similarly, again using the fact that $\IHom(p,-)$ is a $\cC$-module functor which commutes with coequalizers, we have that for each $m \in \cM$ 
the counit map
\begin{equation*}
		F( \IHom(p,m)) \to m
\end{equation*} 
becomes an isomorphism after applying $\IHom(p,-)$. It would follow that the counit map is an isomorphism if we knew that $\IHom(p,-)$ reflects isomorphisms. But this follows directly from Lemma~\ref{lem:fiaghtfulrightadjoint} and the fact that $\IHom(p,-)$ is exact and faithful. 
\end{proof}

If we additionally assume that $\cC$ is rigid, then some of the conditions of the previous theorem are automatically satisfied and become redundant. In particular Cor.~\ref{cor:module-adjoint} implies that when $\cC$ is rigid the functor $\IHom(p,-)$ can always be enhanced to a $\cC$-module functor.  The following lemma shows that in this case there always exists a $\cC$-projective $\cC$-generator.  
\begin{lemma}{\cite[\S 7.10]{egno-book}} \label{lma:Enough_C-projs}
	If $\cC$ is a finite tensor category and $\cM$ is a finite module category in which the action is right exact in $\cC$, then there exists a $\cC$-projective $\cC$-generator.
\end{lemma}   
\begin{proof}
	We claim that if $p \in \cM$ is a projective object (in the ordinary sense) then it is also $\cC$-projective. If this claim holds, then any (ordinary) projective generator of $\cM$ will be a $\cC$-projective $\cC$-generator (see Lemma \ref{lem:gen}), and projective generators are guaranteed to exist since $\cM$ is finite. 
	Thus assume that $p \in \cM$ is projective.  We need to show that $\IHom(p, -)$ is right exact. 
	
	Since $\cC$ is rigid, we have a natural isomorphism of functors:
\begin{equation*}
	\Hom_{\cM}(c \otimes p, m) \cong \Hom_{\cM}(p, {}^*c \otimes m).
\end{equation*}
Taking an object to its dual is a (contravariant) equivalence of categories hence is exact. Moreover since $p$ is projective and the $\cC$-module structure is right exact, we see that the functor $\Hom_{\cM}(- \otimes p, -)$ is left exact in the first variable and right exact in the second variable. Said another way, for each projective $p \in \cM$ we have a right exact functor
\begin{align*}
	G_p: & \; \cM \to \Fun^R(\cC^\op, \Vect) \\
	& m \mapsto (c \mapsto \Hom_{\cM}(c \otimes p, m)).
\end{align*} 
By Proposition \ref{thm:enrichment-of-mod-cats} this functor factors through the Yoneda embedding as
\begin{equation*}
	\IHom(p, -): \cM \to \cC.
\end{equation*}
Thus if $m \to m' \to m'' \to 0$ is an exact sequence in $\cM$ and $x \in \cC$ is any object we have an exact sequence
\begin{equation*}
	\Hom(x, \IHom(p,m)) \to \Hom(x, \IHom(p,m')) \to \Hom(x, \IHom(p,m'')) \to 0.
\end{equation*}
In particular this holds when $x$ is a projective generator (which is guaranteed to exist since $\cC$ is finite). It follows that 
\begin{equation*}
	\IHom(p,m) \to \IHom(p,m') \to \IHom(p,m'') \to 0
\end{equation*}
is an exact sequence: if $x \in \cC$ is a projective generator, then a sequence $y \to y' \to y'' \to 0$ in $\cC$ is exact if and only if 
\begin{equation*}
	\Hom(x, y) \to \Hom(x, y') \to \Hom(x, y'') \to 0
\end{equation*} 
is exact. 
\end{proof}

\begin{proof}[Proof of Thm.~\ref{thm:EGNO2.11.6}]
By Lemma~\ref{lma:Enough_C-projs} and Cor.~\ref{cor:module-adjoint}, the assumptions of Theorem~\ref{thm:C-module-Embedding} are satisfied.
\end{proof}

\begin{corollary} \label{cor:biexact_action}
	Let $\cC$ be a finite tensor category and let $\cM$ be a finite $\cC$-module category. If the action map $\cC \times \cM \to \cM$ is right exact in $\cC$, then the action is in fact exact in each variable separately.  
\end{corollary}

\begin{proof}
	By Lemma~\ref{lem:partially_exact_action} the functor $c \otimes (-)$ is exact for each $c \in \cC$. We must show that $(-) \otimes m$ is exact for each $m \in \cM$. By Theorem~\ref{thm:EGNO2.11.6} there exists an algebra object $A \in \cC$, and an equivalence of $\cC$-module categories $\cM \simeq \Mod{}{A}(\cC)$. Hence there exists a forgetful functor $U:\cM \to \cC$ which is a $\cC$-module functor, is exact, and reflects short exact sequences. It follows that $(-) \otimes m: \cC \to \cM$ is exact if and only if $(-) \otimes U(m): \cC \to \cC$ is exact, but this follows from Lemma~\ref{lma:RigidIsExact}. 
\end{proof}

\begin{remark}
	By taking opposite categories, in the above corollary one may replace `right exact in $\cC$' with `left exact in $\cC$'. 
\end{remark}

\section{Construction of the balanced tensor product} \label{sec:tc-deligne}

In this section we establish the existence of the balanced tensor product of finite module categories over a finite tensor category.  We begin by recalling the definition of the balanced tensor product from \cite{0909.3140}.

\begin{definition}
	Let $\cC$ be a linear monoidal category. 
	Let $\cM$ be a right $\cC$-module category and $\cN$ a left $\cC$-module category.  A bilinear functor $F: \cM \times \cN \to \cL$ is called right exact if it is right exact in each variable.   A {\em $\cC$-balanced functor} into a linear category $\cL$ is a right exact bilinear functor $F: \cM \times \cN \to \cL$ together with a natural isomorphism $F(\otimes^{\cM} \times \id_{\cN}) \cong F(\id_{\cM} \times \otimes^{\cN})$ satisfying the evident pentagon and triangle identities.
	A {\em $\cC$-balanced transformation} is a natural transformation $\eta:F \to G$ of $\cC$-balanced functors such that the following diagram commutes for all $m \in \cM$, $c \in \cC$, and $n \in \cN$:
\begin{center}
\begin{tikzpicture}
	\node (LT) at (0, 1) {$F(m \otimes^{\cM} c, n)$};
	\node (LB) at (0, 0) {$G(m \otimes^{\cM} c, n)$};
	\node (RT) at (3, 1) {$F(m, c \otimes^{\cN} n)$};
	\node (RB) at (3, 0) {$G(m, c \otimes^{\cN} n)$};
	\draw [->] (LT) -- node [left] {$$} (LB);
	\draw [->] (LT) -- node [above] {$$} (RT);
	\draw [->] (RT) -- node [right] {$$} (RB);
	\draw [->] (LB) -- node [below] {$$} (RB);
\end{tikzpicture}.
\end{center}
\end{definition}

\begin{definition}
	Let $\cM$ be a right $\cC$-module category and let $\cN$ be a left $\cC$-module category. The {\em balanced tensor product} is a linear category $\cM \boxtimes_{\cC} \cN$
	 together with a $\cC$-balanced right exact bilinear functor $\boxtimes_{\cC} : \cM \times \cN \to \cM \boxtimes_{\cC} \cN$, such that for all linear categories $\cD$, the functor $\boxtimes_{\cC}$ induces an equivalence between the category of $\cC$-balanced right exact bilinear functors $\cM \times \cN \to \cD$ and the category of right exact linear functors $\cM \boxtimes_{\cC} \cN \to \cD$. 
\end{definition}
\nid More succinctly, we might say, the balanced tensor product $\cM \boxtimes_{\cC} \cN$ corepresents $\cC$-balanced right exact bilinear functors out of $\cM \times \cN$.  The balanced tensor product is also known as the ``relative Deligne tensor product", because the (unbalanced) tensor product $\cM \boxtimes \cN$ of linear categories is often called the ``Deligne tensor product".  

If it exists, the balanced tensor product is unique up to equivalence, and this equivalence is in turn unique up to unique natural isomorphism. Said another way, the 2-category of linear categories representing the balanced tensor product is either contractible or empty. 

Etignof--Nikshych--Ostrik \cite{0909.3140} established the existence of the balanced tensor product of semisimple module categories over semisimple tensor categories over a field of characteristic zero.  A construction of the balanced tensor product for finite tensor categories satisfying the additional assumption that $\cM \boxtimes \cN$ is exact as a $\cC$-bimodule category can be extracted from \cite[Thm 3.1]{1102.3411}.  Note that the proof of the existence of the balanced tensor product outlined in \cite{MR3107567} uses the rigidity and finiteness assumptions in essential ways, but (despite \cite[Note 2.7]{MR3107567}) does not use exactness.  

We give an alternative construction of the balanced tensor product for finite tensor categories.

\begin{theorem} \label{thm:DelignePrdtOverATCExists}
	Let $\cC$ be a finite tensor category and let $\cM_{\cC}$ and ${}_{\cC}\cN$ be finite right and left $\cC$-module categories, respectively. Assume that the action of $\cC$ on $\cM$ and the action of $\cC$ on $\cN$ are right exact in the $\cC$-variable.
	\begin{enumerate}
		\item The balanced tensor product $\cM \boxtimes_{\cC} \cN$ exists and is a finite linear category.
		\item If $\cM = \Mod{A}{}(\cC)$ and $\cN = \Mod{}{B}(\cC)$, then $\cM \boxtimes_{\cC} \cN \simeq \Mod{A }{B}(\cC)$, the category of $A$--$B$-bimodule objects in $\cC$.
		\item The functor $\boxtimes_{\cC}: \cM \times \cN \to \cM \boxtimes_{\cC} \cN$ is exact in each variable and satisfies 
		\begin{equation*}
			 \Hom_{\cM \boxtimes_{\cC} \cN} (x \boxtimes_{\cC} y, x' \boxtimes_{\cC} y') \cong \Hom_{\cC}(1, \IHom_{\cM}(x,x') \otimes \IHom_{\cN}(y, y'))
		\end{equation*}
		\item Given exact $\cC$-module functors $F_0: M \to M'$ and $F_1: N \to N'$, the $\cC$-balanced functor $F: M \times N \to M' \times N' \to M' \boxtimes_{\cC} N'$ induces an exact functor $\overline{F}: M \boxtimes_{\cC} N \to M' \boxtimes_{\cC} N'$.
	\end{enumerate} 
\end{theorem}

\begin{proof}  
	 By Theorem \ref{thm:EGNO2.11.6}, there exist algebra objects $A, B \in \cC$ and equivalences $\cM \simeq \Mod{A}{}(\cC)$ and $\cN \simeq \Mod{}{B}(\cC)$. The linear category $\Mod{A }{B}(\cC)$ is finite by Lemma~\ref{lem:recognizefinitecats} (using the free-forgetful adjunction to $\cC$).   Thus (2) implies (1).
	 
	 By Lemma \ref{lma:RigidIsExact}, the tensor product functor
	\begin{equation*}
		\cM \times \cN \simeq \Mod{A}{}(\cC) \times  \Mod{}{B}(\cC) \to \Mod{A}{B}(\cC)
	\end{equation*}
	is exact in each variable separately. This bilinear functor is also $\cC$-balanced by the associator of $\cC$. Moreover we have
	\begin{align*}
		&\Hom_{\cC}(1, \IHom_{\cM}(x,x') \otimes \IHom_{\cN}(y, y')) \\
		& = \Hom_{\cC}(1, \IHom_{\Mod{A}{}}(x,x') \otimes \IHom_{\cN}(y, y')) \\
		&\cong \Hom_{\cC}(1, \IHom_{\Mod{A}{}}(x,x' \otimes \IHom_{\cN}(y, y')) ) \\
		& \cong \Hom_{\Mod{A}{}}(x, x' \otimes \IHom_{\cN}(y, y') ) \\
		& = \Hom_{\Mod{A}{}}(x, x' \otimes \IHom_{\Mod{}{B}}(y, y') )\\
		& \cong \Hom_{\Mod{A}{}}(x,  \IHom_{\Mod{}{B}}(y, x' \otimes y') )\\
		& \cong \Hom_{\Mod{A}{B}}(x \otimes y, x' \otimes y')
	\end{align*} 
where the second and fifth isomorphisms use the fact that the enriched hom is a $\cC$-module functor (see Cor.~\ref{cor:module-adjoint}). This establishes the formula in (3), and so $(2)$ implies $(3)$. 

	We now prove (2), and then later establish (4).  We wish to show that for any linear category $\cD$ the category of right exact functors 
\begin{equation*}
	\overline{F}:\Mod{A}{B}(\cC) \to \cD
\end{equation*}
	is naturally equivalent to the category of $\cC$-balanced functors $F:\cM \times \cN \to \cD$ that are right exact in each variable separately. Every functor of the former type certainly restricts to one of the later type; we must show that a functor of the latter type extends uniquely (up to canonical isomorphism) to one of the former type. 
	
The key observation is that every object of $\Mod{A}{B}(\cC)$ may be functorially written as a coequalizer of objects in the image of $\cM \times \cN$. Specifically, for any $X \in \Mod{A}{B}(\cC)$, we have the coequalizer:
\begin{equation*}
	{}_A X \otimes B \otimes B_B \rightrightarrows {}_A X \otimes B_B \to {}_A X_B.
\end{equation*}
Let 
$\delta: {}_A X \otimes B \otimes B_B \to {}_A X \otimes B_B$ 
be the difference of the two maps in the coequalizer. 
For any right exact functor $\overline{F}: \Mod{A}{B}(\cC) \to \cD$, the value $\overline{F}({}_A X_B)$ is canonically determined as a cokernel:
\begin{equation*}
	\overline{F}({}_A X_B) = \coker \left( \overline{F}(\delta): \overline{F}( {}_A X \otimes B \otimes B_B) \to \overline{F}({}_A X \otimes B_B) \right).
\end{equation*} 
	
Suppose we are given a $\cC$-balanced functor $F:\cM \times \cN \to \cD$ that is right exact in each variable separately. It is tempting to try to define the extension $\overline{F}: \Mod{A}{B}(\cC) \to \cD$ via a formula of the type:
\begin{equation*}
	``\overline{F}({}_A X_B) := \coker \left( F(\delta): F( {}_A X, B \otimes B_B) \to F({}_A X, B_B) \right)."
\end{equation*} 
The difficulty is that while the relevant objects are in the image of $\cM \times \cN$, the map $\delta$ is not. 
Yet for each $X \in \Mod{A}{B}(\cC)$ we may define a map $\overline{\delta}_X$ as the difference between the composite of the balancing isomorphism and the action of $B$ on $X$,
\begin{equation*}
F( {}_A X, B \otimes B_B) \cong F({}_A X \otimes B, B_B) \to F({}_A X, B_B),
\end{equation*}
and the action of $B$ on itself,
\begin{equation*}
F( {}_A X, B \otimes B_B) \to F({}_A X, B_B).
\end{equation*}
The desired extension can then be defined as $\overline{F}({}_A X_B) := \coker(\overline{\delta}_X)$.
We leave it to the reader to verify that this extension gives a well-defined right exact functor 
\begin{equation*}
	\overline{F}: \Mod{A}{B}(\cC) \to \cD,
\end{equation*} 
and implements the desired equivalence between such right exact functors and $\cC$-balanced exact-in-each-variable functors. Verifying that this construction is well defined makes use of the pentagon identity satisfied by $\cC$-balanced functors. This establishes (1), (2), and (3). 

We now prove the final property (4). By Theorem \ref{thm:EGNO2.11.6}, there exist algebra objects $A, B, A'$, and $B' \in \cC$ and equivalences $\cM \simeq \Mod{A}{}(\cC)$, $\cN \simeq \Mod{}{B}(\cC)$, $\cM' \simeq \Mod{A'}{}(\cC)$, and $\cN \simeq \Mod{}{B'}(\cC)$. Since $F_0$ and $F_1$ are right exact, they are equivalent to tensoring with bimodules:
\begin{align*}
	F_0(-) &\cong {}_{A'}x \otimes_A (-); \\
	F_1(-) & \cong (-) \otimes_B y_{B'}.
\end{align*}
Since $F_0$ and $F_1$ are exact, we may call these modules {\em flat} over $A$ or $B$, respectively. We wish to show that the induced functor:
\begin{equation*}
	\overline{F}(-) = {}_{A'}x \otimes_A (-) \otimes_B y_{B'}: \Mod{A}{B}(\cC) \to \Mod{A'}{B'}(\cC)
\end{equation*}
is exact. Since the forgetful functor $U: \Mod{A'}{B'}(\cC) \to \cC$ is exact and reflects short exact sequences, it is enough to show that 
\begin{equation*}
	x \otimes_A (-) \otimes_B y: \Mod{A}{B}(\cC) \to \cC
\end{equation*}
is exact. Let $0 \to m \to m' \to m'' \to 0$ be a short exact sequence of $A$--$B$-bimodules. After tensoring with $x$ on the left we obtain a sequence of right $B$-modules:
\begin{equation*}
	0 \to x \otimes_A m \to x \otimes_A m' \to x \otimes_A {m''} \to 0
\end{equation*}
Since $x$ is flat, this sequence is exact after forgetting the $B$-module structure. Hence it is also an exact sequence of $B$-modules. Thus, since $y$ is flat, we obtain an exact sequence
\begin{equation*}
		0 \to x \otimes_A m \otimes_B y \to x \otimes_A m'\otimes_B y \to x \otimes_A {m''}  \otimes_B y \to 0
\end{equation*}
	as desired. 
\end{proof}

\begin{remark}
The balanced (Deligne) tensor product, constructed above for finite module categories over a finite tensor category, may fail to exist without the finiteness assumptions---see Lopez Franco \cite{1212.1545} for a counterexample, even in the unbalanced case.  A balanced tensor product does exist a bit more generally, though.  Recall that a finitely cocomplete $k$-additive category is a category that is $\overline{\Vect}_k$-enriched and that admits all finite colimits.  Right exact functors between such categories are, by definition, those that preserve finite colimits.  The balanced ``Kelly" tensor product of finitely cocomplete $k$-additive module categories (over a rigid finitely cocomplete $k$-additive tensor category) corepresents balanced right exact bilinear functors into finitely cocomplete $k$-additive categories.  Unlike the Deligne tensor product, the Kelly tensor product always exists---see \cite{MR651714, MR648793} for the unbalanced case and \cite[Rmk~3.21]{1501.04652} for the balanced case.  Lopez Franco \cite{1212.1545} shows that when the unbalanced Deligne tensor product exist, it is equivalent to the Kelly tensor product.

The proof of the preceding theorem shows that when $\cC$ is a finite tensor category, and $\cM_{\cC}$ and ${}_{\cC}\cN$ are finite right and left $\cC$-module categories, the balanced Deligne tensor product is equivalent to the balanced Kelly tensor product.  This equivalence follows because the proof that $\Mod{A}{B}(\cC)$ corepresents $\cC$-balanced functors into linear categories just as well shows that $\Mod{A}{B}(\cC)$ corepresents $\cC$-balanced functors into finitely cocomplete $k$-additive categories.  (In the displayed commutative diagram, when the target category $\cD$ is not abelian but merely additive, the rows are `exact' in the sense that they are cokernel sequences.)
\end{remark}

\begin{remark}
The construction of the balanced tensor product outlined in~\cite{MR3107567} expresses the tensor $\cM \boxtimes_{\cC} \cN$ as a category of right exact $\cC$-module functors.  The expression~\cite[Eq~13]{MR3107567} for this functor category is cited from~\cite{0909.3140}, and in both places is appropriate for the balanced tensor of module categories but is not quite correct for the balanced tensor of bimodule categories (being off by a twist by a double dual functor).  The functor category that is a balanced tensor product for bimodule categories is $\mathrm{Fun}_{\cC\textrm{-}\mathrm{mod}}^{\mathrm{r.e.}}(\cM^{\ast},\cN)$, where the $\cC$--$\cD$-bimodule category $\cM^{\ast}$ has underlying linear category $\cM^{\mathrm{op}}$, with the left action by an object $c \in \cC$ given by acting on the right by the \emph{left} dual object ${}^{\ast} c \in \cC$ and the right action by an object $d \in \cD$ given by acting on the left by the \emph{left} dual object ${}^{\ast} d \in \cD$.  This functor-category description of the balanced tensor is related to our bimodule description as follows: when $\cM = \Mod{A}{}(\cC)$ and $\cN = \Mod{}{B}(\cC)$, there is a functor equivalence $\Mod{A}{B}(\cC) \ra \mathrm{Fun}_{\cC\textrm{-}\mathrm{mod}}^{\mathrm{r.e.}}(\cM^{\ast},\cN)$ taking a bimodule object ${}_A X_B$ to the functor that takes the object $m \in \cM = \Mod{A}{}(\cC)$ viewed as an object of $\cM^{\ast}$ to the object $m^\ast \otimes_A X \in \Mod{}{B}(\cC) = \cN$, where $m^\ast$ is the right dual of $m$ viewed as an object of $\cC$.
\end{remark}

\begin{example}
Let $\Vect[K]$ denote the category of $K$-graded vector spaces for a finite group $K$.  The balanced product $\Vect \boxtimes_{\Vect[K]} \Vect$ is the category of $k[K]$-bimodules in $\Vect[K]$. As mentioned in Example~\ref{eg:vect}, the category of $k[K]$-modules in $\Vect[K]$ is equivalent to $\Vect$, so the $k[K]$-bimodules in $\Vect[K]$ can be identified with $\text{mod-}k[K] \cong \Rep(K)$.
\end{example}

\begin{remark}
	If ${}_{\cD}\cM_{\cC}$ and ${}_{\cC}\cN_{\cE}$ are bimodule categories, then the actions of $\cD$ and $\cE$ induce a $\cD$--$\cE$-bimodule category structure on $\cM \boxtimes_{\cC} \cN$. This bimodule category satisfies the analogous universal property for $\cC$-balanced bilinear bimodule functors.
\end{remark}


Part (2) of the above theorem expresses the balanced tensor product of two module categories $\cM = \Mod{A}{}(\cC)$ and $\cN = \Mod{}{B}(\cC)$ as a category of bimodules $\cM \boxtimes_\cC \cN \simeq \Mod{A}{B}(\cC)$.  When the tensor category $\cC$ is merely $\Vect$, this expresses the ordinary tensor product $\Mod{}{A}(\Vect) \boxtimes \Mod{}{B}(\Vect)$ of two categories of modules again as a category of modules, $\Mod{}{(A \otimes B)}(\Vect)$.  More generally, the ordinary tensor product of two categories of modules in any tensor categories is again a category of modules, as follows.

\begin{proposition}
	If $\cN = \Mod{}{B}(\cC)$ and $\cP = \Mod{}{C}(\cD)$, for algebra objects $B \in \cC$ and $C \in \cD$, then $$\cN \boxtimes \cP \simeq \Mod{}{(B \boxtimes C)}(\cC \boxtimes \cD)$$ as a left $\cC \boxtimes \cD$-module category.
\end{proposition}

\begin{proof}
	The forgetful functors from $\cN$ to $\cC$ and $\cP$ to $\cD$ are part of monadic adjunctions:
	\begin{align*}
		(-) \otimes B_{B}:\cC \rightleftarrows \cN = \Mod{}{B}(\cC): U \\
		(-) \otimes C_{C}:\cD \rightleftarrows \cP = \Mod{}{C}(\cD): U
	\end{align*}
	Since both functors in the adjunction are right exact (the forgetful functor is exact, not just left exact) these adjunctions descend to an adunction between the tensor products. 
	\begin{equation*}
		(-) \otimes (B \boxtimes C)_{B \boxtimes C}: \cC \boxtimes \cD \rightleftarrows \cN \boxtimes \cP: \overline{U}.
	\end{equation*}
	Any $\cC \boxtimes \cD$-module functor from $\cC \boxtimes \cD$ to itself is given by tensoring by an object in $\cC \boxtimes \cD$, hence any monad on $\cC \boxtimes \cD$ that is compatible with the module actions comes from an algebra in $\cC \boxtimes \cD$.  Hence we need only show that this adjunction is monadic.  
		
	By the the crude monadicity theorem \cite[\S~3.5]{MR771116} we only need to show that $\overline{U}$ reflects isomorphisms and $\overline{U}$ preserves coequalizers of reflexive pairs.  Observe that the individual functors (both called $U$) have these properties. Moreover everything in $\cN$ is a coequalizer of objects in the image of $\cC$, everything in $\cP$ is a coequalizer of objects in the image of $\cD$, and everything in $\cN \boxtimes \cP$ is a coequalizer of objects in the image of $\cN \times \cP$. Thus it follows that every object in $\cN \boxtimes \cP$ is a coequalizer of objects in the image of $\cC \times \cD$. For such objects $\overline{U}$ reflects isomorphisms, since the original $U$ do so. Hence, by the five lemma, it follows that $\overline{U}$ reflects isomorphisms. 

For the latter property we will in fact show that $\overline{U}$ preserves all coequalizers. Since our categories are additive the coequalizer of $f$ and $g$ is the cokernel of $(f-g)$.  Thus it is sufficient to show that $\overline{U}$ is exact and hence that $\overline{U}$ preserves cokernels.  The exactness of $\overline{U}$ follows from Part (4) of Theorem~\ref{thm:DelignePrdtOverATCExists} and exactness of the original forgetful functors $U$.  
\end{proof}

\bibliographystyle{alpha}
\bibliography{dtcbib}
\end{document}

%% file: RDTP-Preamble.tex
\usepackage{amsmath}       
\usepackage{amsthm}        
\usepackage{amssymb}       
\usepackage{amsfonts}
\usepackage[all]{xy}
\usepackage{xspace}
\usepackage{calc}
\usepackage{comment}
\usepackage{pgfplots}
\usepgflibrary{fpu}

\usepackage{ifthen}		

\newcommand\ignore[1]{}

\DeclareRobustCommand{\SkipTocEntry}[5]{}

\usepackage{tikz}
\usetikzlibrary{matrix}
\usetikzlibrary{arrows,backgrounds,patterns,scopes,external,
    decorations.pathreplacing,
    decorations.pathmorphing
}
\newlength{\fuzzwidth}
\newlength{\arrowlength}
\newlength{\arrowwidth}
\newlength{\pointrad}
\newlength{\linewid}
\newlength{\circlerad}
\newlength{\smcirclerad}
\newcommand{\fuzzcolor}{black!25}
\newcommand{\arrowcolor}{black!25}
\newcommand{\covercolor}{black!0}
\newcommand{\graycolor}{black!55}
\newcommand{\graylightcolor}{black!40}

\newcommand{\coverwidthfuzz}{6pt}
\newcommand{\coverwidth}{3.5pt}
\newcommand{\coverwidththin}{3.25pt}
\newcommand{\coverwidththick}{3.75pt}

\newlength{\linewidthin}
\newlength{\linewidthick}
\tikzset{
	coverline/.style={
	preaction={draw,line width=\coverwidth,\covercolor}}, 
	coverlinethin/.style={
	preaction={draw,line width=\coverwidththin,\covercolor}}, 
	coverlinethick/.style={
	preaction={draw,line width=\coverwidththick,\covercolor}}, 
	coverlineleft/.style={
	preaction={draw,line width=\coverwidthfuzz,\covercolor,decorate,decoration={curveto,amplitude=0,raise=.35*\fuzzwidth}}}, 
coverlinelefttail/.style={
	preaction={draw,line width=\coverwidthfuzz,\covercolor,decorate,decoration={curveto,amplitude=0,raise=.35*\fuzzwidth,pre=moveto,pre length=2pt}}}, 
        fuzzlefttail/.style={
        preaction={draw,line width=\fuzzwidth,\fuzzcolor,decorate,decoration={curveto,pre=moveto,pre length=2pt,amplitude=0,raise=.5*\fuzzwidth}}}, 
        linestylethin/.style={line width=\linewidthin},
        linestylethick/.style={line width=\linewidthick},
        linestylegray/.style={line width=\linewid,\graycolor},
        linestylegraylight/.style={line width=\linewid,\graylightcolor}
}
\tikzset{
        fuzzright/.style={
        preaction={draw,line width=\fuzzwidth,\fuzzcolor,decorate,decoration={curveto,amplitude=0,raise=-.5*\fuzzwidth}}},
        fuzzleft/.style={
        preaction={draw,line width=\fuzzwidth,\fuzzcolor,decorate,decoration={curveto,amplitude=0,raise=.5*\fuzzwidth}}},
        fuzzrightpre/.style={ 
        preaction={draw,line width=2pt,\fuzzcolor,decorate,decoration={curveto,amplitude=0,raise=-1pt,pre=moveto,pre length=12pt}}},
        fuzzleftpre/.style={ 
        preaction={draw,line width=2pt,\fuzzcolor,decorate,decoration={curveto,post=moveto,post length=32pt,amplitude=0,raise=1pt}}},        
        outstyle/.style={\arrowcolor, line width=\arrowwidth},
        linestyle/.style={line width=\linewid}
}

\newcommand{\arxiv}[1]{\href{http://arxiv.org/abs/#1}{\tt arXiv:\nolinkurl{#1}}}

\newcommand{\googlebooks}[1]{(preview at \href{http://books.google.com/books?id=#1}{google books})}

\theoremstyle{plain} 

\newtheorem{theorem}{Theorem}[section]
\newtheorem*{theorem*}{Theorem}
\newtheorem{lemma}[theorem]{Lemma}
\newtheorem{corollary}[theorem]{Corollary}          
\newtheorem{proposition}[theorem]{Proposition}

\theoremstyle{definition} 
\newtheorem{definition}[theorem]{Definition}

\theoremstyle{remark}  
\newtheorem{remark}[theorem]{Remark}
\newtheorem{example}[theorem]{Example}

\newtheorem{warning}[theorem]{{Warning}}

\newtheoremstyle{special_statement} 
	{\topskip}
	{\topskip}
	{\addtolength{\leftskip}{2.5em} \itshape }
	{}
	{\bfseries}
	{:}
	{.5em}
	{}
\theoremstyle{special_statement}

\newcommand{\Mod}[2]  
{
  \ifthenelse{\equal{#1}{}}{  			
		\ifthenelse{\equal{#2}{}}		
			{\mathrm{Mod}}{ 			
				{\mathrm{Mod}\textrm{-}#2}		
			}
	}{									
		\ifthenelse{\equal{#2}{}}		
			{{#1\textrm{-}\mathrm{Mod}}}{		
				{{#1\textrm{-}\mathrm{Mod}\textrm{-}#2}}	
			}
	}
}

\newcommand{\nid}{\noindent}
\newcommand{\ra}{\rightarrow}

\DeclareMathOperator{\Hom}{Hom}

\DeclareMathOperator{\IHom}{\underline{Hom}}
\DeclareMathOperator{\Fun}{Fun}

\DeclareMathOperator{\coeq}{coeq}
\DeclareMathOperator{\eq}{eq}

\DeclareMathOperator*{\coker}{coker}

\renewcommand{\mp}{\mathrm{mp}}
\newcommand{\op}{\mathrm{op}}

\newcommand{\id}{\mathrm{id}}

\newcommand{\Vect}{\mathrm{Vect}}

\newcommand{\Rep}{\mathrm{Rep}}

\def\cC{\mathcal C}\def\cD{\mathcal D}
\def\cE{\mathcal E}\def\cF{\mathcal F}\def\cG{\mathcal G}
\def\cK{\mathcal K}\def\cL{\mathcal L}
\def\cM{\mathcal M}\def\cN{\mathcal N}\def\cP{\mathcal P}
\def\cR{\mathcal R}
\def\cX{\mathcal X}

\definecolor{CSPcolor}{rgb}{0.0,0.5,0.75}	
\definecolor{NScolor}{rgb}{0.5,0.0,0.5}		
\definecolor{CDcolor}{rgb}{0.8,0.0,0.2}		

\newcommand{\CD}[1]{\marginpar{\vspace*{-20pt}\tiny\color{CDcolor}{ #1}\vspace*{20pt}}}
\newcommand{\NS}[1]{\marginpar{\vspace*{-20pt}\tiny\color{NScolor}{ #1}\vspace*{20pt}}}

\usepackage[pdftex, colorlinks=true, linkcolor=black, citecolor=black, urlcolor=black
	]{hyperref}

%% file: RDTP.bbl
\newcommand{\etalchar}[1]{$^{#1}$}
\newcommand{\noopsort}[1]{}\def\cprime{$'$} \def\cprime{$'$} \def\cprime{$'$}
\begin{thebibliography}{EGNO15}

\bibitem[BK01]{MR1797619}
Bojko Bakalov and Alexander Kirillov, Jr.
\newblock {\em Lectures on tensor categories and modular functors}, volume~21
  of {\em University Lecture Series}.
\newblock AMS, Providence, RI, 2001.

\bibitem[BW85]{MR771116}
Michael Barr and Charles Wells.
\newblock {\em Toposes, triples and theories}, volume 278 of {\em Grundlehren
  der Mathematischen Wissenschaften [Fundamental Principles of Mathematical
  Sciences]}.
\newblock Springer-Verlag, New York, 1985.

\bibitem[BZBJ15]{1501.04652}
David Ben-Zvi, Adrien Brochier, and David Jordan.
\newblock Integrating quantum groups over surfaces: quantum character varieties
  and topological field theory.
\newblock Preprint, \arxiv{1501.04652}, 2015.

\bibitem[Del90]{MR1106898}
Pierre Deligne.
\newblock Cat\'egories {T}annakiennes.
\newblock In {\em The {G}rothendieck {F}estschrift, {V}ol.\ {II}}, volume~87 of
  {\em Progr. Math.}, pages 111--195. Birkh\"auser Boston, Boston, MA, 1990.
\newblock \mathscinet{MR1106898}.

\bibitem[DN13]{MR3107567}
Alexei Davydov and Dmitri Nikshych.
\newblock The {P}icard crossed module of a braided tensor category.
\newblock {\em Algebra Number Theory}, 7(6):1365--1403, 2013.
\newblock \mathscinet{MR3107567}, \arxiv{1202.0061}.

\bibitem[DNO13]{MR3022755}
Alexei Davydov, Dmitri Nikshych, and Victor Ostrik.
\newblock On the structure of the {W}itt group of braided fusion categories.
\newblock {\em Selecta Math. (N.S.)}, 19(1):237--269, 2013.
\newblock \mathscinet{MR3022755} \arxiv{1109.5558}.

\bibitem[DSPS13]{DTCI}
Christopher~L. Douglas, Chris Schommer-Pries, and Noah Snyder.
\newblock Dualizable tensor categories.
\newblock Preprint, \arxiv{1312.7188}, 2013.

\bibitem[EGH{\etalchar{+}}11]{MR2808160}
Pavel Etingof, Oleg Golberg, Sebastian Hensel, Tiankai Liu, Alex Schwendner,
  Dmitry Vaintrob, and Elena Yudovina.
\newblock {\em Introduction to representation theory}, volume~59 of {\em
  Student Mathematical Library}.
\newblock American Mathematical Society, Providence, RI, 2011.
\newblock With historical interludes by Slava Gerovitch.

\bibitem[EGNO15]{egno-book}
Pavel Etingof, Shlomo Gelaki, Dmitri Nikshych, and Victor Ostrik.
\newblock {\em Tensor categories}, volume 205 of {\em Mathematical Surveys and
  Monographs}.
\newblock American Mathematical Society, Providence, RI, 2015.

\bibitem[ENO10]{0909.3140}
Pavel Etingof, Dmitri Nikshych, and Victor Ostrik.
\newblock Fusion categories and homotopy theory.
\newblock {\em Quantum Topol.}, 1(3):209--273, 2010.
\newblock With an appendix by Ehud Meir. \arxiv{0909.3140}.

\bibitem[EO04]{EO-ftc}
Pavel Etingof and Viktor Ostrik.
\newblock Finite tensor categories.
\newblock {\em Mosc. Math. J.}, 4(3):627--654, 782--783, 2004.
\newblock \arxiv{math.QA/0301027}.

\bibitem[Fre64]{MR0166240}
Peter Freyd.
\newblock {\em Abelian categories. {A}n introduction to the theory of
  functors}.
\newblock Harper's Series in Modern Mathematics. Harper \& Row Publishers, New
  York, 1964.
\newblock \mathscinet{MR0166240}.

\bibitem[FSV13]{MR3063919}
J{\"u}rgen Fuchs, Christoph Schweigert, and Alessandro Valentino.
\newblock Bicategories for {B}oundary {C}onditions and for {S}urface {D}efects
  in 3-d {TFT}.
\newblock {\em Comm. Math. Phys.}, 321(2):543--575, 2013.
\newblock \mathscinet{MR3063919} \arxiv{1203.4568}.

\bibitem[{Gre}11]{1102.3411}
Justin {Greenough}.
\newblock {Relative centers and tensor products of tensor and braided fusion
  categories}.
\newblock 2011.
\newblock Preprint, \arxiv{1102.3411}.

\bibitem[GS12a]{1202.4396}
Pinhas Grossman and Noah Snyder.
\newblock The {B}rauer-{P}icard group of the {A}saeda-{H}aagerup fusion
  categories, 2012.
\newblock \arxiv{1202.4396}.

\bibitem[GS12b]{MR2909758}
Pinhas Grossman and Noah Snyder.
\newblock Quantum subgroups of the {H}aagerup fusion categories.
\newblock {\em Comm. Math. Phys.}, 311(3):617--643, 2012.
\newblock \arxiv{1102.2631}, \mathscinet{MR2909758}.

\bibitem[JFS17]{MR3590516}
Theo Johnson-Freyd and Claudia Scheimbauer.
\newblock (Op)lax natural transformations, twisted quantum field theories, and ``even higher" Morita categories.
\newblock {\em Adv. Math.}, 307:147--223,  2017.
\newblock \arxiv{1502.06526}, \mathscinet{MR3590516}.

\bibitem[JL09]{MR2511638}
David Jordan and Eric Larson.
\newblock On the classification of certain fusion categories.
\newblock {\em J. Noncommut. Geom.}, 3(3):481--499, 2009.
\newblock \mathscinet{MR2511638} \arxiv{0812.1603}.

\bibitem[Kel82a]{MR651714}
Gregory~Maxwell Kelly.
\newblock {\em Basic concepts of enriched category theory}, volume~64 of {\em
  London Mathematical Society Lecture Note Series}.
\newblock Cambridge University Press, Cambridge-New York, 1982.

\bibitem[Kel82b]{MR648793}
Gregory~Maxwell Kelly.
\newblock Structures defined by finite limits in the enriched context. {I}.
\newblock {\em Cahiers Topologie G{\'e}om. Diff{\'e}rentielle}, 23(1):3--42,
  1982.
\newblock Third Colloquium on Categories, Part VI (Amiens, 1980).

\bibitem[{Lop}12]{1212.1545}
Ignacio {Lopez Franco}.
\newblock {Tensor products of finitely cocomplete and abelian categories}.
\newblock 2012.
\newblock Preprint, \arxiv{1212.1545}.

\bibitem[Mac71]{MR0354798}
Saunders MacLane.
\newblock {\em Categories for the working mathematician}.
\newblock Springer-Verlag, New York, 1971.
\newblock \mathscinet{MR0354798}.

\bibitem[M{\"u}g03]{MR1966524}
Michael M{\"u}ger.
\newblock From subfactors to categories and topology. {I}. {F}robenius algebras
  in and {M}orita equivalence of tensor categories.
\newblock {\em J. Pure Appl. Algebra}, 180(1-2):81--157, 2003.
\newblock \arxiv{math.CT/0111204}.

\bibitem[Ost03]{MR1976459}
Victor Ostrik.
\newblock Module categories, weak {H}opf algebras and modular invariants.
\newblock {\em Transform. Groups}, 8(2):177--206, 2003.
\newblock \arxiv{math.QA/0111139}.

\bibitem[Tam01]{tambara}
Daisuke Tambara.
\newblock A duality for modules over monoidal categories of representations of
  semisimple {H}opf algebras.
\newblock {\em J. Algebra}, 241(2):515--547, 2001.

\end{thebibliography}
